\documentclass[12pt,a4paper,reqno]{amsart}
\usepackage{amsmath}
\usepackage{amssymb}
\usepackage{amsfonts}

\newtheorem{theorem}{Theorem}
\theoremstyle{plain}

\newtheorem{definition}{Definition}[section]
\newtheorem{example}{Example}[section]

\newtheorem{lemma}{Lemma}[section]

\newtheorem{proposition}{Proposition}[section]
\newtheorem{remark}{Remark}[section]

\numberwithin{equation}{section} \oddsidemargin.25cm
\input{tcilatex}

\begin{document}
\title[Local energy decay]{Local energy decay for the wave equation with a
nonlinear time dependent damping.}
\author{Ahmed BCHATNIA}
\address{A. BCHATNIA, Department of Mathematics, Faculty of Sciences of
Tunis, University of Tunis El Manar, Tunisia and University of DAMMAM, KSA}
\email{ahmed.bchatnia@fst.rnu.tn}
\author{Moez DAOULATLI}
\address{M. DAOULATLI, Department of Mathematics, Faculty of Sciences of
Bizerte, University of Carthage, Tunisia}
\email{moez.daoulatli@infcom.rnu.tn }
\date{\today }

\begin{abstract}
This paper addresses a wave equation on a exterior domain in $\mathbb{R}^{d}$
($d$ odd) with nonlinear time dependent dissipation. Under a microlocal
geometric condition we prove that the decay rates of the local energy
functional are obtained by solving a nonlinear non-autonomous differential
equation.
\end{abstract}

\maketitle

\section{Introduction and Statement of the result}

Let $O$ be a compact domain of $%
\mathbb{R}
^{d}$ $\left( d\geq 3\text{ is odd}\right) $ with $C^{\infty }$ boundary $%
\Gamma =\partial \Omega $ and $\Omega =\mathbb{R}^{d}\setminus O,$ $O\subset
B_{R}$ for some $R>0$. Consider the following wave equation with a nonlinear
time dependent damping 
\begin{equation}
\left\{ 
\begin{array}{lc}
\partial _{t}^{2}u-\Delta u+a\left( x\right) \rho \left( t\right) g\left(
\partial _{t}u\right) =0, & \text{in }\mathbb{R}_{+}\times \Omega , \\ 
u=0, & \text{on }\mathbb{R}_{+}\times \Gamma , \\ 
u\left( 0,x\right) =\varphi _{0}\quad \text{ and }\quad \partial _{t}u\left(
0,x\right) =\varphi _{1}. & 
\end{array}%
\right.  \label{sys:nonlinear}
\end{equation}%
Here $\Delta $ denotes the Laplace operator in the space variables.

The nonlinear terms satisfy:

\begin{itemize}
\item $a\left( x\right) $ is a non-negative function in $C^{\infty }\left(
\Omega \right) $ with compact support such that supp$a\subset B_{R}$.

\item $\rho $ is a positive, monotone and differentiable function on $%
\mathbb{R}
_{+}:$ there exists a positive constant $C_{0}>0$ such that%
\begin{equation*}
\begin{array}{l}
\left\vert \rho ^{\prime }\left( t\right) \right\vert \leq C_{0}\rho \left(
t\right) \text{, for all }t\geq 0.%
\end{array}%
\end{equation*}%
Moreover, without loss of generality we assume that $\rho (0)=1.$

\item $g:\mathbb{R}\rightarrow \mathbb{R}$ is a continuous and monotone
increasing function with $g(0)=0$.
\end{itemize}

The natural space of initial data is $H=H_{D}\left( \Omega \right) \times
L^{2}\left( \Omega \right) $ which is the completion of $(C_{0}^{\infty
}\left( \Omega \right) )^{2}$\ with respect to the norm 
\begin{equation*}
\left\Vert \varphi \right\Vert _{H}^{2}=\left\Vert (\varphi _{0},\varphi
_{1})\right\Vert _{H}^{2}=\frac{1}{2}\int_{\Omega }\left\vert \nabla \varphi
_{0}\right\vert ^{2}+\left\vert \varphi _{1}\right\vert ^{2}dx.
\end{equation*}%
It is known that (see Lions--Strauss \cite{lions strauss}) under the
conditions above, the system $\left( \ref{sys:nonlinear}\right) $\ is well
posed in the space $H$, i.e., for any initial state $(u_{0},u_{1})\in H$
there exists a unique weak solution of $\left( \ref{sys:nonlinear}\right) $
such that%
\begin{equation*}
u\in C^{0}\left( 
\mathbb{R}
_{+},H_{D}\left( \Omega \right) \right) ;\text{ }\partial _{t}u\in
C^{0}\left( 
\mathbb{R}
_{+},L^{2}\left( \Omega \right) \right) .
\end{equation*}%
For every $t\in \mathbb{R}_{+}$, we define the evolution operator $U(t)$ by 
\begin{equation*}
\begin{array}{ccccc}
U(t): & H & \longrightarrow & H &  \\ 
& \left( u_{0},u_{1}\right) & \longmapsto & U\left( t\right) \left(
u_{0},u_{1}\right) & =\left( u\left( t\right) ,\partial _{t}u\left( t\right)
\right) ,%
\end{array}%
\end{equation*}%
where $u$ is the solution of $\left( \text{\ref{sys:nonlinear}}\right) $.
Let us consider the energy at instant $t$ defined by%
\begin{eqnarray*}
E_{u}\left( t\right) &=&\frac{1}{2}\int_{\Omega }\left( \left\vert \nabla
u\left( t,x\right) \right\vert ^{2}+\left\vert \partial _{t}u\left(
t,x\right) \right\vert ^{2}\right) dx \\
&=&\left\Vert U\left( t\right) \varphi \right\Vert _{H}^{2}.
\end{eqnarray*}%
We formally obtain the following identity 
\begin{equation}
E_{u}\left( T\right) +\int_{0}^{T}\int_{\Omega }a\left( x\right) g\left(
\partial _{t}u\right) \partial _{t}u\text{ }dx\text{ }dt=E_{u}\left(
0\right) ,  \label{energy inequality}
\end{equation}%
for every $T\geq 0$. We define the local energy by%
\begin{eqnarray*}
E_{r}\left( u\right) \left( t\right) &=&\frac{1}{2}\int_{\Omega \cap
B_{r}}\left( \left\vert \nabla u\left( t,x\right) \right\vert
^{2}+\left\vert \partial _{t}u\left( t,x\right) \right\vert ^{2}\right) dx \\
&=&\left\Vert \left( u\left( t\right) ,\partial _{t}u\left( t\right) \right)
\right\Vert _{H\left( B_{r}\right) }^{2}
\end{eqnarray*}%
where $B_{r}=\left\{ x\in 
\mathbb{R}
^{d},\left\vert x\right\vert <r\right\} $, contains the obstacle $O$.

Our goal is to give the rate of decay of the local energy. Throughout the
paper we will frequently invoke the following notation%
\begin{equation*}
\Omega _{s,t}=[s,t]\times \Omega ,\text{ }t\geq s\geq 0\text{ and }\Omega
_{0,t}=\Omega _{t}.
\end{equation*}%
For convenience we also introduce the following weighted measure on $\Omega $%
: 
\begin{equation*}
\mathfrak{m}_{a}=a\left( x\right) dxdt.
\end{equation*}

The problem of the local energy decay for the wave equation and systems on
exterior domain has been intensively investigated during the last decades.
For the classical wave equation, the story goes up to the pioneering works
of Lax-Phillips \cite{la-ph}, and Morawetz, Strauss and Ralston \cite%
{mor-rals-strauss}. When the obstacle is trapping, Ralston \cite{ral} proved
that there is no uniform decay rate, and Morawetz-Ralston-Strauss \cite%
{mor-rals-strauss} and Melrose \cite{mel} obtained the exponential decay for
nontrapping obstacle. In \cite{burq}, without any assumption on the dynamics
Burq proved the logarithmic decay of the local energy with respect to any
Sobolev norm larger than the initial energy. Nakao in \cite{nakao1} proved
that the local energy decay exponentially if $d$ is odd and polynomially if $%
d$ is even under the Lions's geometric condition. Later, in \cite{alkh} and
for general obstacles, Aloui and Khenissi proved the exponential decay of
the local energy by mean of linear internal localized damping and Khenissi 
\cite{kh} proved the polynomial decay in even space dimension. For that,
they introduced the exterior geometric control condition (E.G.C.) inspired
from the so-called microlocal condition of Bardos-Lebeau-Rauch \cite{blr}
and they used in a crucial way the propagation properties of the microlocal
defect measures of G\'{e}rard \cite{ge1} (see also Lebeau \cite{lebeau}).
More recently in \cite{DAOU}, using a nonlinear internal localized damping,
Daoulatli obtained various decay rates depending on the behavior of the
damping term. Concerning the semilinear waves on unbounded domains, we quote
the work of Bchatnia and Daoulatli \cite{bcda}, which establishes an
exponential decay of the local energy for the solutions of subcritical wave
equation outside convex obstacle. We also mention the result of Daoulatli et
al \cite{mbm}\ on the energy decay rates of \ the local energy for the
elastic system with a nonlinear damping. Finally, we quote the following
results on the energy decay rates for the wave equation with time dependent
damping in bounded domain \cite{bella, daou 2, mart, mess, nakao}. However,
concerning the decay property of the local energy of wave equation with
nonlinear time dependent dissipation in exterior domains no results seem to
be known.

Now, we recall the exterior geometric condition of \cite{alkh}.

\begin{definition}[EGC]
Let $R>0$ such that $O\subset B_{R}$, $T_{R}>0$ and $\omega =\left\{ x\in
\Omega ;a\left( x\right) >0\right\} $. We say that $\left( \omega
,T_{R}\right) $ verifies the exterior geometric control condition on $B_{R}$ 
$\left( E.G.C\right) $, if every geodesic $\gamma $ starting from $B_{R}$ at
time $t=0$, is such that

$\cdot $ $\gamma $ leaves $%
\mathbb{R}
_{+}\times B_{R}$ before the time $T_{R}$, or

$\cdot $ $\gamma $ meets $%
\mathbb{R}
_{+}\times \omega $ between the times $0$ and $T_{R}$.
\end{definition}

In this paper, under the condition $\left( \text{EGC}\right) $, we give the
rate of decay of the local energy of solutions of the wave equation with
nonlinear time dependent dissipation in exterior domain. More precisely, the
rate of decay will be determined from the following nonlinear,
non-autonomous ODE:%
\begin{equation*}
\frac{dS}{dt}+q\left( t,S\left( t\right) \right) =0,\qquad S\left( 0\right)
=E_{u}\left( 0\right) ,
\end{equation*}%
where the function $q$ is defined in Subsection \ref{sec:aux functions}
below.

\subsection{Behavior of the dissipation at the origin and infinity\label%
{sec:h0}}

In order to characterize decay rates for the energy, we need to introduce
several special functions, which in turn will depend on the growth of $g$
near the origin and near infinity. For that purpose let $m_{0}\geq 1$, and
following \cite{las-tun}, we classify the behavior of $g$ \textit{near the
Origin}:

\begin{description}
\item[AO1] Linearly bounded on $I$:%
\begin{equation*}
\frac{1}{m_{0}}y^{2}\leq g\left( y\right) y\leq m_{0}y^{2},\quad y\in I,
\end{equation*}

\item[AO2] Superlinear on $I$:%
\begin{equation*}
g\left( y\right) y\leq m_{0}y^{2},\quad y\in I,
\end{equation*}

\item[AO3] Sublinear on $I$:%
\begin{equation*}
\frac{1}{m_{0}}y^{2}\leq g\left( y\right) y,\quad y\in I,
\end{equation*}%
where $I=\left[ -\eta ,\eta \right] $ and $0<\eta <1$.
\end{description}

We then define a concave function $h_{0}$ which will describe the growth of $%
g$ near the origin. Since $g$ is a continuous monotone increasing function
vanishing at zero, following \cite{las-tat} there exists a concave monotone
increasing function $h_{0}$ defined on $%
\mathbb{R}
_{+}$ such that $h_{0}(0)=0$ and 
\begin{equation}
h_{0}\left( g\left( y\right) y\right) \geq \epsilon _{0}\left( g\left(
y\right) ^{2}+y^{2}\right) \quad \text{ for }\quad \left\vert y\right\vert
<\eta _{0},  \label{h0 definition}
\end{equation}%
for some $\epsilon _{0}$,\ $\eta _{0}>0$. For example when $g$ is
superlinear and odd, then $h_{0}^{-1}\left( s\right) =\sqrt{s}g\left( \sqrt{s%
}\right) $ when $\left\vert s\right\vert \leq \eta $. For further details on
the construction of such function we refer the interested reader to \cite%
{las-tat,daou 2,mid}.

We now assume that $g$ is linearly bounded at infinity:%
\begin{equation}
\frac{1}{m}y^{2}\leq g\left( y\right) y\leq my^{2},\quad \left\vert
y\right\vert \geq \eta _{0}.  \label{linearly definition}
\end{equation}%
We introduce some auxiliary functions $\alpha \left( t\right) $ and $\beta
\left( t\right) $ which are linked to the function $\rho $ as follows$:$%
\begin{equation*}
\begin{array}{l}
\beta \left( t\right) =\left\{ 
\begin{array}{ll}
\begin{array}{l}
\medskip \frac{1}{T}\rho \left( t+T\right)%
\end{array}
& \text{if }\rho \text{ is decreasing} \\ 
\left. 
\begin{array}{ll}
\medskip \frac{1}{T} & t<T \\ 
\medskip \frac{1}{T}\rho \left( t-T\right) & t\geq T%
\end{array}%
\right\} & \text{if }\rho \text{ is increasing}%
\end{array}%
\right.%
\end{array}%
\end{equation*}%
and%
\begin{equation*}
\begin{array}{l}
\alpha \left( t\right) =\left\{ 
\begin{array}{ll}
1 & \text{if }\rho \text{ is decreasing} \\ 
\rho ^{-2}\left( t+T\right) & \text{if }\rho \text{ is increasing,}%
\end{array}%
\right.%
\end{array}%
\end{equation*}%
here $T$ is a positive constant which will be precised in the statement of
Theorem 1.

\subsection{Auxiliary functions\label{sec:aux functions}}

Let $h_{0},$ $\alpha $ and $\beta $ be as defined above and set%
\begin{equation}
h=I+\mathfrak{m}_{a}\left( \Omega _{T}\right) h_{0}\circ \tfrac{I}{T%
\mathfrak{m}_{a}\left( \Omega _{T}\right) }.  \label{def:h}
\end{equation}%
We introduce, for $t\geq 0$,%
\begin{equation}
\begin{array}{l}
q\left( t,.\right) =\beta \left( t\right) h^{-1}\circ \frac{\alpha \left(
t\right) }{K}I,%
\end{array}
\label{def: q}
\end{equation}%
where $K$ is a positive constant such that $K\geq C_{T}$, here $C_{T}$ is
the constant that appears in the estimate $\left( \ref{ctdefinition}\right)
. $ We note that $C_{T}$ is independent of the initial data.

\section{The main result}

In this section we give the main result of this paper.

\begin{theorem}
\label{t:1}Let $T_{R}~$be such that $\left( \left\{ x\in \Omega ;a\left(
x\right) >0\right\} ,T_{R}\right) $ satisfies the exterior geometric
condition on $B_{R}$. Then there exist $T\geq T_{R}+9R$ and a positive
constant $C_{T}$ such that inequality%
\begin{equation}
E_{R}\left( u\right) \left( t\right) \leq S\left( t-T\right) ,\qquad \text{%
for all }t\geq T,  \label{energy decay formula}
\end{equation}%
holds for every solution $u$ of system (\ref{sys:nonlinear}) if the initial
data $\left( u_{0},u_{1}\right) $ in the energy space $H$ is compactly
supported in $B_{R}$. Here $S\left( t\right) $ is the solution of the
following nonlinear differential equation: 
\begin{equation}
\frac{dS}{dt}+q\left( t,S\left( t\right) \right) =0,\qquad S\left( 0\right)
=E_{u}\left( 0\right) ,  \label{sharp ODE}
\end{equation}%
and the function $q$ is defined in Subsection \ref{sec:aux functions} above.
Moreover, if for some $T_{0}\gg 1$%
\begin{equation}
\int_{T_{0}}^{t}q\left( s,\gamma \right) ds\underset{t\rightarrow +\infty }{%
\longrightarrow }\infty ,  \label{over under damping condition}
\end{equation}%
for every $0<\gamma \ll 1$, then%
\begin{equation*}
E_{R}\left( u\right) \left( t\right) \underset{t\rightarrow +\infty }{%
\longrightarrow }0.
\end{equation*}
\end{theorem}

The proof of theorem will be stated in Section \ref{rate of decay}. In the
next section we give some applications of this result.

\section{Applications}

We recall first the following lemma useful to us to determine the rate of
decay.

\begin{lemma}[\protect\cite{daou 2}]
\label{lemma ode}

\begin{enumerate}
\item Let $\alpha _{1}$ is a positive, differentiable and decreasing
function on $%
\mathbb{R}
_{+}$ and $\beta _{1}$ is a non-negative function on $%
\mathbb{R}
_{+}.$ Let $S$ be a positive function verifying the following differential
inequality 
\begin{equation}
\frac{dS}{dt}+\beta _{1}\left( t\right) p\left( \alpha _{1}\left( t\right)
S\right) \leq 0,\text{ }S\left( 0\right) >0.  \label{equ: lemma ode}
\end{equation}%
We assume that $p$ is a strictly increasing function on $\left[ 0,\alpha
_{1}\left( 0\right) S\left( 0\right) \right] $ with $p\left( 0\right) =0$
and verifies%
\begin{equation*}
\begin{array}{ll}
p\left( x\right) \leq m^{-1}x & \text{for all }x\leq \alpha _{1}\left(
0\right) S\left( 0\right) \text{ and for some }m>0.%
\end{array}%
\end{equation*}%
Then we have%
\begin{equation}
S\left( t\right) \leq \frac{1}{\alpha _{1}\left( t\right) }\psi ^{-1}\left(
\int_{0}^{t}\alpha _{1}\left( s\right) \beta _{1}\left( s\right) ds-m\ln
\left( \frac{\alpha _{1}\left( t\right) }{\alpha _{1}\left( 0\right) }%
\right) \right) ,\forall t\geq 0,  \label{ode solution}
\end{equation}%
where%
\begin{equation*}
\psi \left( x\right) =\int_{x}^{\alpha _{1}\left( 0\right) S\left( 0\right) }%
\frac{ds}{p\left( s\right) }.
\end{equation*}

\item If in addition the function $p$ satisfies the following property%
\begin{equation}
p\left( \alpha _{1}\left( t\right) x\right) \geq mp\left( x\right) p\left(
\alpha _{1}\left( t\right) \right) ,\text{ }\forall t\geq 0,\text{ and }x\in %
\left[ 0,S\left( 0\right) \right]   \label{Lem:p lower bound}
\end{equation}%
for some $m>0$ and $p$ is a strictly increasing function on $\left[
0,S\left( 0\right) \right] $, then the function $S$ of $\left( \text{\ref%
{equ: lemma ode}}\right) $, verifies%
\begin{equation*}
S\left( t\right) \leq \psi ^{-1}\left( \int_{0}^{t}mp\left( \alpha
_{1}\left( s\right) \right) \beta _{1}\left( s\right) ds\right) ,\forall
t\geq 0
\end{equation*}%
where%
\begin{equation*}
\psi \left( x\right) =\int_{x}^{S\left( 0\right) }\frac{ds}{p\left( s\right) 
}.
\end{equation*}
\end{enumerate}
\end{lemma}

\begin{remark}
The rate of decay of the energy depends on $\alpha ,$ $\beta $ and the
behavior of $h^{-1}$ near zero. To determine it, we only have to find $%
0<\epsilon _{0}\leq 1$, such that%
\begin{equation*}
\chi \left( s\right) \leq h^{-1}\left( s\right) \text{, for every }s\leq
\epsilon _{0},
\end{equation*}%
with%
\begin{equation*}
\begin{array}{l}
\chi \left( s\right) =C_{1}h_{0}^{-1}\left( \frac{s}{2C_{2}}\right) \text{,
for every }s\leq \epsilon _{0},%
\end{array}%
\end{equation*}%
where 
\begin{equation*}
C_{1}=\min \left( T\mathfrak{m}_{a}\left( \Omega \right) ,1\right) \text{
and }C_{2}=\max \left( T\mathfrak{m}_{a}\left( \Omega \right) ,1\right) .
\end{equation*}
\end{remark}

In the sequel $c$ denotes a positive constant which is independent of the
initial data. Moreover, if $\rho $ is increasing, we suppose that there
exist $t_{0}$, $c_{0}>0$, such that 
\begin{equation}
\begin{array}{l}
\rho \left( t-2T\right) \geq c_{0}\rho \left( t\right)%
\end{array}%
\text{ for every }t\geq t_{0}.  \label{ro increasing infinity condition}
\end{equation}

\subsection{The linear case}

Let $g\left( s\right) =s$, $0\leq |s|<1$. According to (\ref{h0 definition}%
), the auxiliary function $h$ is defined as $h(y)=2y$. Then 
\begin{equation*}
\psi \left( x\right) =2\ln \left( \frac{\left\Vert \varphi \right\Vert
_{H}^{2}}{x}\right)
\end{equation*}

\begin{enumerate}
\item $\rho $ is decreasing: $\alpha _{1}\left( t\right) =\frac{1}{K}$ and $%
\beta _{1}\left( t\right) =\frac{_{1}}{T}\rho \left( t+T\right) $. As a
result, using Theorem \ref{t:1} and some computations, the estimate 
\begin{equation*}
E_{R}\left( u\right) \left( t\right) \leq S(t-T),\qquad t\geq T
\end{equation*}%
becomes%
\begin{equation*}
\begin{array}{c}
E_{R}\left( u\right) \left( t\right) \leq c\left\Vert \varphi \right\Vert
_{H}^{2}\exp \left( -\frac{1}{KT}\int_{0}^{t}\rho \left( s\right) ds\right) ,%
\text{ for all }t\geq 0.%
\end{array}%
\end{equation*}

\item $\rho $ is increasing: $\alpha _{1}\left( t\right) =\frac{1}{K\rho
^{2}\left( t+T\right) }$ and $\beta _{1}\left( t\right) =\frac{1}{T}\rho
\left( t-T\right) $, $t\geq T$. Then%
\begin{equation*}
\begin{array}{c}
E_{R}\left( u\right) \left( t\right) \leq c\left\Vert \varphi \right\Vert
_{H}^{2}\exp \left( -\frac{1}{KT}\int_{2T}^{t}\frac{\rho \left( s-2T\right) 
}{\rho ^{2}\left( s\right) }ds\right) ,\text{for all }t\geq 2T.%
\end{array}%
\end{equation*}%
Now, using $\left( \text{\ref{ro increasing infinity condition}}\right) $
and making some arrangement, we obtain%
\begin{equation*}
\begin{array}{c}
E_{R}\left( u\right) \left( t\right) \leq c\left\Vert \varphi \right\Vert
_{H}^{2}\exp \left( -\frac{c_{0}}{KT}\int_{0}^{t}\rho ^{-1}\left( s\right)
ds\right) ,\text{ for all }t\geq 0.%
\end{array}%
\end{equation*}
\end{enumerate}

\begin{remark}
An important special case of $\left( \text{\ref{sys:nonlinear}}\right) $ is
when $\rho \left( t\right) =\left( 1+t\right) ^{\tau }$, $\tau \in 
\mathbb{R}
$. We have%
\begin{equation*}
\begin{array}{ll}
\medskip E_{R}\left( u\right) \left( t\right) \leq cE_{u}\left( 0\right)
\exp \left( -c_{K}\left( 1+t\right) ^{1-\left\vert \tau \right\vert }\right)
& \left\vert \tau \right\vert <1, \\ 
\medskip E_{R}\left( u\right) \left( t\right) \leq cE_{u}\left( 0\right)
\left( 1+t\right) ^{-\frac{1}{KT}} & \tau =-1, \\ 
\medskip E_{R}\left( u\right) \left( t\right) \leq cE_{u}\left( 0\right)
\left( 1+t\right) ^{-\frac{c_{0}}{KT}} & \tau =1,%
\end{array}%
\end{equation*}%
for every $t\geq 0$. Moreover, it is clear that, we cannot obtain the decay
to zero of the energy when $\left\vert \tau \right\vert >1$.
\end{remark}

\subsection{The nonlinear case}

Since $g$ is linearly bounded near infinity, there exists a constant $A>0$,
such that%
\begin{equation}
\begin{array}{l}
T\mathfrak{m}_{a}\left( \Omega \right) h_{0}^{-1}\left( \frac{s}{2T\mathfrak{%
m}_{a}\left( \Omega \right) }\right) \leq h^{-1}\left( s\right) \leq s\text{%
, for every }s\leq A.%
\end{array}
\label{h-1 approx g linear near infinity}
\end{equation}%
So the rate of decay of the energy depends only on the behavior of $g$ near
the origin.

\begin{example}[Superlinear polynomial damping near the origin]
\label{example:superlinear at origin} Suppose $g\left( s\right)
=s|s|^{r_{0}-1}$, $0\leq \left\vert s\right\vert <1$ for some $r_{0}>1$. The
auxiliary function $h_{0}$ which may be defined as%
\begin{equation*}
h_{0}^{-1}\left( s\right) =s^{\frac{1+r_{0}}{2}},\text{ for }s\in \left[ 0,1%
\right] .
\end{equation*}%
Consequently, we obtain%
\begin{equation*}
\begin{array}{l}
\psi \left( x\right) \leq \frac{2}{r_{0}-1}\left( x^{\frac{1-r_{0}}{2}%
}-\left\Vert \varphi \right\Vert _{H}^{1-r_{0}}\right) .%
\end{array}%
\end{equation*}

\begin{enumerate}
\item $\rho $ decreasing. Let $\alpha _{1}\left( t\right) =\left( 2TK\right)
^{-1}$ and $\beta _{1}\left( t\right) =\rho \left( \left( t+1\right)
T\right) $. Then we deduce that%
\begin{equation*}
\begin{array}{c}
E_{R}\left( u\right) \left( t\right) \leq \left\Vert \varphi \right\Vert
_{H}^{2}\left( 1+\left( \frac{\left\Vert \varphi \right\Vert _{H}^{2}}{2TK}%
\right) ^{\frac{r_{0}-1}{2}}\frac{r_{0}-1}{4KT}\int_{T}^{t}\rho \left(
s\right) ds\right) ^{-\frac{2}{r_{0}-1}};\qquad t\geq T.%
\end{array}%
\end{equation*}%
After some arrangement and by choosing $K$ big enough, we obtain%
\begin{equation*}
\begin{array}{c}
E_{R}\left( u\right) \left( t\right) \leq c\left\Vert \varphi \right\Vert
_{H}^{2}\left( 1+\left( \frac{\left\Vert \varphi \right\Vert _{H}^{2}}{K}%
\right) ^{\frac{r_{0}-1}{2}}\frac{c}{KT}\int_{0}^{t}\rho \left( s\right)
ds\right) ^{-\frac{2}{r_{0}-1}},\text{ }t\geq 0.%
\end{array}%
\end{equation*}

\item $\rho $ increasing. Let $\alpha _{1}\left( t\right) =\left( 2TK\rho
^{2}\left( t+T\right) \right) ^{-1}$ and $\beta _{1}\left( t\right) =\rho
\left( t-T\right) $. Then,%
\begin{equation*}
\begin{array}{c}
E_{R}\left( u\right) \left( t\right) \leq c\left\Vert \varphi \right\Vert
_{H}^{2}\left( 1+\left( \frac{\left\Vert \varphi \right\Vert _{H}^{2}}{K}%
\right) ^{\frac{r_{0}-1}{2}}\frac{c}{KT}\int_{2T}^{t}\frac{\rho \left(
s-2T\right) }{\left( \rho \left( s\right) \right) ^{r_{0}+1}}ds\right) ^{-%
\frac{2}{r_{0}-1}},%
\end{array}%
\end{equation*}%
for $t\geq 2T$.

Using $\left( \text{\ref{ro increasing infinity condition}}\right) $ and
after some computation, we obtain%
\begin{equation*}
\begin{array}{c}
E_{R}\left( u\right) \left( t\right) \leq c\left\Vert \varphi \right\Vert
_{H}^{2}\left( 1+\left( \frac{\left\Vert \varphi \right\Vert _{H}^{2}}{K}%
\right) ^{\frac{r_{0}-1}{2}}\frac{c}{KT}\int_{0}^{t}\left( \rho \left(
s\right) \right) ^{-r_{0}}ds\right) ^{-\frac{2}{r_{0}-1}},\text{ }t\geq 0.%
\end{array}%
\end{equation*}
\end{enumerate}
\end{example}

\begin{remark}
Take $\rho \left( t\right) =\left( 1+t\right) ^{\tau }$, $\tau \in \left[ -1,%
\frac{1}{r_{0}}\right] $. We have 
\begin{equation*}
\begin{array}{ll}
\medskip E_{R}\left( u\right) \left( t\right) \leq C_{K}\left( \ln \left(
2+t\right) \right) ^{-\frac{2}{r_{0}-1}}, & \tau =-1\text{ or }\tau =\frac{1%
}{r_{0}}, \\ 
E_{R}\left( u\right) \left( t\right) \leq C_{K}\left( 1+t\right) ^{\mu }, & 
\tau \in \left] -1,\frac{1}{r_{0}}\right[ ,%
\end{array}%
\end{equation*}%
for every $t\geq 0$, with%
\begin{equation*}
\begin{array}{ll}
\medskip \mu =-\frac{2}{r_{0}-1}\left( 1+\tau \right) , & -1<\tau \leq 0, \\ 
\mu =-\frac{2}{r_{0}-1}\left( 1-\tau r_{0}\right) , & 0\leq \tau <\frac{1}{%
r_{0}}.%
\end{array}%
\end{equation*}
\end{remark}

\begin{example}[Sublinear near the origin]
Assume $g\left( s\right) =s\left\vert s\right\vert ^{\theta _{0}-1}$, $0\leq
\left\vert s\right\vert <1$, $\theta _{0}\in \left( 0,1\right) $. We have%
\begin{equation*}
h_{0}^{-1}(s)=s^{\frac{1+\theta _{0}}{2\theta _{0}}},\text{ for }s\in \left[
0,1\right] .
\end{equation*}%
Consequently, we infer that 
\begin{equation*}
\begin{array}{l}
\psi \left( x\right) \leq \frac{2\theta _{0}}{1-\theta _{0}}\left( x^{-\frac{%
1-\theta _{0}}{2\theta _{0}}}-\left\Vert \varphi \right\Vert _{H}^{-\frac{%
1-\theta _{0}}{\theta _{0}}}\right) .%
\end{array}%
\end{equation*}

\begin{enumerate}
\item $\rho $ decreasing. $\alpha _{1}\left( t\right) =\left( 2TK\right)
^{-1}$ and $\beta _{1}\left( t\right) =\rho \left( t+T\right) $. Similarly
as in example 1, we obtain%
\begin{equation*}
\begin{array}{c}
E_{R}\left( u\right) \left( t\right) \leq c\left\Vert \varphi \right\Vert
_{H}^{2}\left( 1+\left( \frac{\left\Vert \varphi \right\Vert _{H}^{2}}{K}%
\right) ^{\left( 1-\theta _{0}\right) /2\theta _{0}}\frac{c}{KT}%
\int_{0}^{t}\rho \left( s\right) ds\right) ^{-2\theta _{0}/\left( 1-\theta
_{0}\right) },\ t\geq 0.%
\end{array}%
\end{equation*}

\item $\rho $ increasing. $\alpha _{1}\left( t\right) =\left( 2TK\rho
^{2}\left( \left( t+1\right) T\right) \right) ^{-1}$ and $\beta _{1}\left(
t\right) =\rho \left( t-T\right) $. Then,$\ $for all $t\geq 0$%
\begin{equation*}
\begin{array}{c}
E_{R}\left( u\right) \left( t\right) \leq c\left\Vert \varphi \right\Vert
_{H}^{2}\left( 1+\left( \frac{\left\Vert \varphi \right\Vert _{H}^{2}}{K}%
\right) ^{\left( 1-\theta _{0}\right) /2\theta _{0}}\frac{c}{KT}%
\int_{0}^{t}\left( \rho \left( s\right) \right) ^{-1/\theta _{0}}ds\right)
^{-2\theta _{0}/\left( 1-\theta _{0}\right) }.%
\end{array}%
\end{equation*}
\end{enumerate}
\end{example}

\begin{remark}
Take $\rho \left( t\right) =\left( 1+t\right) ^{\tau }$, $\tau \in \left[
-1,\theta _{0}\right] $. We have%
\begin{equation*}
\begin{array}{ll}
\medskip E_{R}\left( u\right) \left( t\right) \leq C_{K}\left( \ln \left(
2+t\right) \right) ^{-2\theta _{0}/\left( 1-\theta _{0}\right) }, & \tau =-1%
\text{ or }\tau =\theta _{0}, \\ 
E_{R}\left( u\right) \left( t\right) \leq C_{K}\left( 1+t\right) ^{\mu }, & 
\tau \in \left] -1,\theta _{0}\right[ ,%
\end{array}%
\end{equation*}%
for every $t\geq 0$, with%
\begin{equation*}
\begin{array}{ll}
\medskip \mu =-\frac{2\theta _{0}}{1-\theta _{0}}\left( 1+\tau \right) , & 
-1<\tau \leq 0, \\ 
\mu =-\frac{2\theta _{0}}{1-\theta _{0}}\left( 1-\frac{\tau }{\theta _{0}}%
\right) , & 0\leq \tau <\theta _{0}.%
\end{array}%
\end{equation*}
\end{remark}

\begin{example}[Exponential damping at the origin]
\label{example:exp at origin} $g\left( s\right) =se^{-1/s^{2}}$, $%
0<\left\vert s\right\vert <1$. We take 
\begin{equation}
h_{0}^{-1}(s)=se^{-1/s},\text{ for }s\in \left[ 0,1\right]
\label{Ex: h 0 -1 exp origin def}
\end{equation}

We assume that $\rho $ is decreasing. Let $\alpha _{1}\left( t\right)
=\left( 2TK\right) ^{-1}$ and $\beta _{1}\left( t\right) =\rho \left(
t+T\right) $. Then 
\begin{equation*}
\psi \left( x\right) \leq \left\Vert \varphi \right\Vert _{H}^{2}\left( \exp
\left( \frac{1}{x}\right) -\exp \left( \frac{2TK}{\left\Vert \varphi
\right\Vert _{H}^{2}}\right) \right) .
\end{equation*}%
Consequently, we find%
\begin{equation*}
\begin{array}{c}
E_{R}\left( u\right) \left( t\right) \leq 2TK\left[ \ln \left( \frac{1}{%
2KT\left\Vert \varphi \right\Vert _{H}^{2}}\int_{T}^{t}\rho \left( s\right)
ds+\exp \left( \frac{2TK}{\left\Vert \varphi \right\Vert _{H}^{2}}\right)
\right) \right] ^{-1},\quad t\geq T,%
\end{array}%
\end{equation*}%
and we deduce that%
\begin{equation*}
\begin{array}{l}
E_{R}\left( u\right) \left( t\right) \leq cK\left( \ln \left( \frac{1}{%
2KT\left\Vert \varphi \right\Vert _{H}^{2}}\int_{0}^{t}\rho \left( s\right)
ds+2\right) \right) ^{-1},\qquad t\geq 0.%
\end{array}%
\end{equation*}
\end{example}

\begin{remark}
Take $\rho \left( t\right) =\left( 1+t\right) ^{\tau }$, $\tau \in \left[
-1,0\right] $. We have 
\begin{equation*}
\begin{array}{ll}
\medskip E_{R}\left( u\right) \left( t\right) \leq C_{K}\left( \ln \left( 
\frac{\ln \left( 1+t\right) }{KT}+2\right) \right) ^{-1}, & \tau =-1, \\ 
E_{R}\left( u\right) \left( t\right) \leq C_{K}\left( \ln \left( \frac{1+t}{%
KT}+2\right) \right) ^{-1}, & \tau \in \left] -1,0\right] ,%
\end{array}%
\end{equation*}%
for every $t\geq 0$.
\end{remark}

\section{Lax-Phillips theory and preliminary results\label{lax-phil}}

This section is devoted to some results on the Lax-Phillips Theory \cite%
{la-ph}, which are useful for the definition and the essential properties of
the Lax-Phillips semi-group.

Let us consider the free wave equation 
\begin{equation}
\left\{ 
\begin{array}{lc}
\partial _{t}^{2}u-\Delta u=0, & \text{in }\mathbb{R\times R}^{d}, \\ 
u\left( 0,.\right) =\varphi _{0}\text{\ , }\partial _{t}u\left( 0,.\right)
=\varphi _{1}. & 
\end{array}%
\right. .  \label{sys lib}
\end{equation}%
\ We recall that the solution of $\left( \text{\ref{sys lib}}\right) $ is
given by the propagator%
\begin{equation}
U_{0}\left( t\right) :H_{0}\ni \varphi =\left( \varphi _{0},\varphi
_{1}\right) \rightarrow U_{0}\left( t\right) \varphi =\left( u,\partial
_{t}u\right) \in H_{0}.  \label{linear propagator}
\end{equation}%
where $H_{0}$ is the \ completion\ of $(C_{0}^{\infty }(\mathbb{R}^{d}))%
{{}^2}%
$ with respect to the norm%
\begin{equation*}
\left\Vert \varphi \right\Vert _{H_{0}}^{2}=\frac{1}{2}\int_{\mathbb{R}%
^{d}}(\left\vert \nabla \varphi _{1}\right\vert 
{{}^2}%
+\left\vert \varphi _{2}\right\vert 
{{}^2}%
)dx.
\end{equation*}%
Following Lax and Phillips \cite{la-ph}, we denote : 
\begin{equation*}
D_{+}^{0}=\left\{ \varphi =\left( \varphi _{0},\varphi _{1}\right) \in H_{0}%
\text{ ; }U_{0}(t)\varphi =0\text{ on }\left\vert x\right\vert <t\text{ },%
\text{ }t\geq 0\right\} ,
\end{equation*}%
the space of outgoing data, and 
\begin{equation*}
D_{-}^{0}=\left\{ \varphi =\left( \varphi _{0},\varphi _{1}\right) \in H_{0}%
\text{ ; }U_{0}(t)\varphi =0\text{ on }\left\vert x\right\vert <-t\text{ },%
\text{ }t\leq 0\right\} ,
\end{equation*}%
the space of incoming data associated to the solutions of $\left( \text{\ref%
{sys lib}}\right) $. We consider the wave equation in the exterior domain $%
\Omega $.%
\begin{equation}
\left\{ 
\begin{array}{lc}
\partial _{t}^{2}u-\Delta u=0\text{\ } & 
\mathbb{R}
_{+}\times \Omega , \\ 
u=0 & 
\mathbb{R}
_{+}\times \Gamma , \\ 
u\left( 0,x\right) =\varphi _{0}\text{ and }u_{t}\left( 0,x\right) =\varphi
_{0}. & 
\end{array}%
\right.  \label{Sys lin ext}
\end{equation}%
We denote $U_{D}(t)$ the linear wave group, defining the solution of $\left( 
\text{\ref{Sys lin ext}}\right) $%
\begin{equation*}
\begin{array}{ccccc}
U_{D}(t): & H & \longrightarrow & H &  \\ 
& \left( \varphi _{0},\varphi _{1}\right) & \longmapsto & U_{D}\left(
t\right) \left( \varphi _{0},\varphi _{1}\right) & =\left( u\left( t\right)
,\partial _{t}u\left( t\right) \right)%
\end{array}%
\end{equation*}%
Let us consider the wave equation in exterior domain%
\begin{equation}
\left\{ 
\begin{array}{lc}
\partial _{t}^{2}u-\Delta u+a\left( x\right) \partial _{t}u=0 & \text{in }%
\mathbb{R}_{+}\times \Omega , \\ 
u=0 & \text{on }\mathbb{R}_{+}\times \Gamma , \\ 
u\left( 0,x\right) =\varphi _{0}\quad \text{ and }\quad \partial _{t}u\left(
0,x\right) =\varphi _{1}, & 
\end{array}%
\right. ,  \label{sys linear}
\end{equation}%
where $\left( \varphi _{0},\varphi _{1}\right) \in H$.

We denote $U_{L}(t)$ the linear wave group, defining the solution of $\left( 
\text{\ref{sys linear}}\right) $%
\begin{equation}
\begin{array}{ccccc}
U_{L}(t): & H & \longrightarrow & H &  \\ 
& \left( \varphi _{0},\varphi _{1}\right) & \longmapsto & U_{L}\left(
t\right) \left( \varphi _{0},\varphi _{1}\right) & =\left( u\left( t\right)
,\partial _{t}u\left( t\right) \right)%
\end{array}
\label{linear exterior propagator}
\end{equation}

We choose $R>0$ such that $B_{R}$ contains the obstacle $O$. Then we define
spaces of outgoing and incoming data associated to solutions of problem $%
\left( \text{\ref{Sys lin ext}}\right) $ by%
\begin{equation}
D_{+}^{R}=\left\{ \varphi =\left( \varphi _{0},\varphi _{1}\right) \in H%
\text{ ; }U_{D}(t)\varphi =0\text{ on }\left\vert x\right\vert <t+R,\text{ }%
t\geq 0\right\} ,  \label{d + definition}
\end{equation}%
\begin{equation}
D_{-}^{R}=\left\{ \varphi =\left( \varphi _{0},\varphi _{1}\right) \in H%
\text{ ; }U_{D}(t)\varphi =0\text{ on }\left\vert x\right\vert <-t+R,\text{ }%
t\leq 0\right\} .  \label{d- definition}
\end{equation}%
These spaces satisfy the following properties:

\begin{enumerate}
\item $D_{+}^{R}$ and $D_{-}^{R}\ $are closed in $H$.

\item $D_{+}^{R}$ and $D_{-}^{R}$ are orthogonal and%
\begin{equation}
D_{+}^{R}\oplus D_{-}^{R}\oplus \left( \left( D_{+}^{R}\right) ^{\perp }\cap
\left( D_{-}^{R}\right) ^{\perp }\right) =H.  \label{decom energy space}
\end{equation}
\end{enumerate}

\begin{remark}
$\left. {}\right. $

\begin{enumerate}
\item Solutions of \ $\left( \text{\ref{Sys lin ext}}\right) $ and $\left( 
\text{\ref{sys:nonlinear}}\right) $ verify the finite speed propagation
property.

\item The nonlinearity being localized, it is easy to see that 
\begin{equation}
U\left( t\right) =U_{D}\left( t\right) \text{ on }D_{+}^{R}\text{ for every }%
t\geq 0.  \label{solution operating}
\end{equation}

\item Following \cite{la-ph}, we denote by $P_{+}\left( \text{resp.}%
P_{-}\right) $ the orthogonal projection of $H$ onto the orthogonal
complement of $D_{+}^{R}\left( \text{resp.}D_{-}^{R}\right) $. Thanks to $%
\left( \text{\ref{decom energy space}}\right) $, we easily deduce 
\begin{equation}
P_{+}\varphi \in \left( D_{+}^{R}\right) ^{\perp }\cap \left(
D_{-}^{R}\right) ^{\perp }\text{ if }\varphi \in \left( D_{-}^{R}\right)
^{\perp }.  \label{p+ property}
\end{equation}

\item The semi-group $U\left( t\right) $ operates on $D_{+}^{R}$ for $t\geq
0 $. Using the fact that the Cauchy problem admits a unique solution, we
obtain: 
\begin{eqnarray}
U\left( t\right) \varphi &=&U\left( t\right) P_{+}\varphi +U\left( t\right)
\left( I-P_{+}\right) \varphi  \notag \\
&=&U\left( t\right) P_{+}\varphi +U_{D}\left( t\right) \left( I-P_{+}\right)
\varphi ,  \label{d+ decomposition}
\end{eqnarray}%
for every $\varphi $ in $H$ and for every $t\in \mathbb{R}_{+}$.
\end{enumerate}
\end{remark}

We denote by $K=\left( D_{+}^{R}\right) ^{\perp }\cap \left(
D_{-}^{R}\right) ^{\perp }$, and we define the nonlinear Lax-Phillips
operator on $K$ by$\ $%
\begin{equation}
Z\left( t\right) =P_{+}U\left( t\right) P_{-}\text{ \ for }t\geq 0.
\label{lax phillips semigroup def}
\end{equation}%
In order to prove that $Z\left( t\right) $ operates on $K$, we need the
following lemma.

\begin{lemma}
$\left. {}\right. $\newline
\label{lemma operance}Let $\left( \varphi ,\psi \right) \in H\times H$\emph{%
\ }and\emph{\ }$t\geq 0$,\emph{\ }we have%
\begin{equation}
\left\langle U\left( t\right) \varphi ,\psi \right\rangle _{H}-\left\langle
\varphi ,U_{D}\left( -t\right) \psi \right\rangle _{H}=-\int_{0}^{t}\rho
\left( s\right) \left\langle ag(\partial _{t}u(s)),\partial _{t}v\left(
s-t\right) \right\rangle _{L^{2}\left( \Omega \right) }ds,
\label{D- operating neumann intermediaire}
\end{equation}%
where we denoted by $U(t)\varphi =(u(t)$,$\partial _{t}u(t))$ and $%
U_{D}(t)\psi =(v(t),\partial _{t}v(t))$.
\end{lemma}

\begin{proof}
Noting that for each $\left( u_{0},u_{1}\right) \in H$ the solution $u$ of $%
\left( \ref{sys:nonlinear}\right) $ are given as a limit of smooth solution $%
u_{n}$ with initial data $\left( u_{n,0},u_{n,1}\right) $ smooth such that $%
\left( u_{n,0},u_{n,1}\right) \longrightarrow \left( u_{0},u_{1}\right) $ in 
$H$. Note that $\left\Vert u_{n}\left( t,.\right) -u\left( t,.\right)
\right\Vert _{H_{D}}+\left\Vert \partial _{t}u_{n}\left( t,.\right)
-\partial _{t}u\left( t,.\right) \right\Vert _{L^{2}}\longrightarrow 0$,
uniformly on $%
\mathbb{R}
_{+}.$ So we may assume that $u$ is a smooth function. Let $\varphi \in H$
and $\psi \in H$. Thanks to Green formula%
\begin{eqnarray*}
\frac{d}{dt}\left\langle U\left( t\right) \varphi ,U_{D}(t)\psi
\right\rangle _{H} &=&\frac{d}{dt}\left( \left\langle \nabla u,\nabla
v\right\rangle _{L^{2}}+\left\langle \partial _{t}u,\partial
_{t}v\right\rangle _{L^{2}}\right) \\
&=&\left\langle \partial _{t}^{2}u-\Delta u,\partial _{t}v\right\rangle
_{L^{2}}+\left\langle \partial _{t}u,\partial _{t}^{2}v-\Delta
v\right\rangle _{L^{2}} \\
&=&\left\langle \partial _{t}^{2}u-\Delta u,\partial _{t}v\right\rangle
_{L^{2}}
\end{eqnarray*}%
We easily deduce%
\begin{equation*}
\left\langle U(t)\varphi ,U_{D}(t)\psi \right\rangle _{H}-\left\langle
\varphi ,\psi \right\rangle _{H}=-\int_{0}^{t}\rho \left( s\right)
\left\langle ag(\partial _{t}u(s)),\partial _{t}v\left( s\right)
\right\rangle _{L^{2}}ds.
\end{equation*}%
Consequently, we obtain%
\begin{equation*}
\left\langle U(t)\varphi ,\psi \right\rangle _{H}-\left\langle \varphi
,U_{D}(-t)\psi \right\rangle _{H}=-\int_{0}^{t}\rho \left( s\right)
\left\langle ag(\partial _{t}u(s)),\partial _{t}v\left( s-t\right)
\right\rangle _{L^{2}}ds.
\end{equation*}
\end{proof}

Using the result above we prove that the Lax-Phillips semi-group operates on 
$K$.

\begin{proposition}
\label{propositionz(t)}$\left. {}\right. $\newline
The semi-group $\left( Z\left( t\right) \right) _{t\geq 0}$ operates on $K$.
\end{proposition}

\begin{proof}
Let $\varphi \in K$, and $t\geq 0$. According to $\left( \text{\ref{p+
property}}\right) $, it suffices to verify that $U\left( t\right) \varphi
\in \left( D_{-}^{R}\right) ^{\perp }$. Let $\left( \varphi ,\psi \right)
\in \left( D_{-}^{R}\right) ^{\perp }\times D_{-}^{R}$, then $\left( \text{%
\ref{D- operating neumann intermediaire}}\right) $ yields,%
\begin{equation*}
\left\langle U\left( t\right) \varphi ,\psi \right\rangle _{H}=\left\langle
\varphi ,U_{D}\left( -t\right) \psi \right\rangle _{H}-\int_{0}^{t}\rho
\left( s\right) \left\langle ag(\partial _{t}u),\partial _{t}v\left(
s-t\right) \right\rangle _{L^{2}\left( \Omega \right) }ds.
\end{equation*}%
Knowing that $U_{D}\left( s-t\right) $ operates on $D_{-}^{R}$ for $s\leq t$%
, and thus $U_{D}\left( s-t\right) \varphi =0$ for $\left\vert x\right\vert
\leq R+t-s$, we deduce that%
\begin{equation*}
\partial _{t}v\left( s-t\right) _{\mid B_{R}}=0\text{ on }\left[ 0,t\right] .
\end{equation*}%
This gives,%
\begin{equation*}
\left\langle U\left( t\right) \varphi ,\psi \right\rangle _{H}=\left\langle
\varphi ,U_{D}\left( -t\right) \psi \right\rangle _{H}=0.
\end{equation*}
\end{proof}

\section{Proof of the main Theorem\label{rate of decay}}

To prove the main Theorem we need some preliminary results.

\subsection{Preliminary results}

First we give the following result due to Aloui and Khenissi \cite{alkh}.

\begin{proposition}[\protect\cite{alkh}]
\label{proposition ALKH}We assume that (EGC) holds. Then, there exist $%
c_{0}>0$ and $T>0$ such that 
\begin{equation}
\left\Vert Z_{L}\left( t\right) \right\Vert _{L\left( K\right) }=\left\Vert
P_{+}U_{L}\left( t\right) P_{-}\right\Vert _{L\left( K\right) }\leq c_{0}<1,
\label{Z_L contraction}
\end{equation}%
for every $t\geq T$.
\end{proposition}

In the proposition below we prove a mixed observability.

\begin{proposition}
There exist $T>0$ and $C>0$ such that for every $\varphi $ in $K$ , we have%
\begin{eqnarray}
\left\Vert Z\left( t\right) \varphi \right\Vert _{H}^{2} &\leq &C\left(
\left\Vert Z\left( t\right) \varphi \right\Vert _{H}^{2}-\left\Vert Z\left(
t+T\right) \varphi \right\Vert _{H}^{2}\right)  \notag \\
&&+C\int_{t}^{t+T}\int_{\Omega }\rho \left( s\right) g\left( \partial
_{t}u\right) \partial _{t}v\left( s-t\right) +\partial _{t}u\partial
_{t}v\left( s-t\right) d\mathfrak{m}_{a},  \label{mixed observability}
\end{eqnarray}%
for all $t\geq 0$, where $u$ and $v$ denote respectively the solution of $%
\left( \text{\ref{sys:nonlinear}}\right) $ and $\left( \text{\ref{sys linear}%
}\right) $ with initial data $\varphi =\left( \varphi _{0},\varphi
_{1}\right) $ and $\left( \left( v_{0},v_{1}\right) =Z\left( t\right)
\varphi \right) $\ in $K$.
\end{proposition}

\begin{proof}
Let $\varphi \in K$. According to Proposition \ref{propositionz(t)}, $%
Z\left( t\right) \varphi \in K$ for every $t\geq 0$. Setting $\left( \phi
\left( s\right) ,\partial _{s}\phi \left( s\right) \right) =U_{t}\left(
s\right) Z\left( t\right) \varphi $, where $\phi $ is the solution of 
\begin{equation}
\begin{array}{l}
\left\{ 
\begin{array}{ll}
\partial _{s}^{2}\phi -\Delta \phi +\rho \left( s+t\right) a\left( x\right)
g\left( \partial _{s}\phi \right) =0, & 
\mathbb{R}
_{+}\times \Omega , \\ 
\phi =0, & 
\mathbb{R}
_{+}\times \Gamma , \\ 
\left( \phi \left( 0\right) ,\partial _{s}\phi \left( 0\right) \right)
=Z\left( t\right) \varphi , & 
\end{array}%
\right. 
\end{array}
\label{system phi}
\end{equation}%
we have 
\begin{equation*}
U_{t}\left( s\right) Z\left( t\right) \varphi -U_{D}\left( s\right) \left(
P_{+}-I\right) U\left( t\right) \varphi =U\left( s+t\right) \varphi ,
\end{equation*}%
which implies in particular that 
\begin{equation*}
U_{t}\left( s\right) Z\left( t\right) \varphi =U\left( s+t\right) \varphi ,%
\text{ on }B_{R},
\end{equation*}%
and 
\begin{equation}
P_{+}U_{t}\left( s\right) Z\left( t\right) \varphi =P_{+}U\left( s+t\right)
\varphi =Z\left( t+T\right) \varphi .  \label{zt estimate}
\end{equation}%
Then we obtain%
\begin{equation*}
\left\Vert Z\left( t+T\right) \varphi \right\Vert _{H}\leq \left\Vert
U_{t}\left( s\right) Z\left( t\right) \varphi -U_{L}(T)Z\left( t\right)
\varphi \right\Vert _{H}+\left\Vert Z_{L}\left( T\right) Z\left( t\right)
\varphi \right\Vert _{H}.
\end{equation*}%
According to Proposition \ref{proposition ALKH}, there exist $0<c_{0}<1$ and 
$T>0$ such that,%
\begin{equation}
\left\Vert Z\left( t+T\right) \varphi \right\Vert _{H}\leq \left\Vert
U_{t}\left( s\right) Z\left( t\right) \varphi -U_{L}(T)Z\left( t\right)
\varphi \right\Vert _{H}+c_{0}\left\Vert Z\left( t\right) \varphi
\right\Vert _{H}.  \label{contrac inequ}
\end{equation}%
On the other hand, let $v$ solution of $\left( \text{\ref{sys linear}}%
\right) $ with the same initial data $Z\left( t\right) \varphi $ in $K$.
Then we define $z=\phi -v$, which satisfies the following system 
\begin{equation*}
\left\{ 
\begin{array}{ll}
\partial _{t}^{2}z-\Delta z+a\left( x\right) \rho \left( s+t\right) g\left(
\partial _{t}\phi \right) -a(x)\partial _{t}v=0, & 
\mathbb{R}
_{+}\times \Omega , \\ 
z=0, & 
\mathbb{R}
_{+}\times \partial \Omega , \\ 
\left( z\left( 0\right) ,\partial _{t}z\left( 0\right) \right) =0. & 
\end{array}%
\right. 
\end{equation*}%
Since $Z\left( t\right) \varphi \in K$, then $a\left( x\right) \left( \rho
\left( s+t\right) g\left( \partial _{t}\phi \right) -\partial _{t}v\right)
\in L^{2}\left( \left( 0,T\right) \times \Omega \right) $. This observation
permits us to apply energy identity, whence 
\begin{equation}
E_{z}(T)=\int_{\Omega _{T}}a\left( x\right) \left( \partial _{t}v-\rho
\left( s+t\right) g\left( \partial _{t}\phi \right) \right) \partial
_{t}z\;dxdt  \label{z}
\end{equation}%
The monotonicity of $g$ and $g(0)=0$ gives $g(s)s\geq 0$ for all $s\in 
\mathbb{R}
.$ Therefore using the identity $\left( \ref{z}\right) $, we get 
\begin{equation*}
\left\Vert U_{t}\left( s\right) Z\left( t\right) \varphi -U_{L}(T)Z\left(
t\right) \varphi \right\Vert _{H}^{2}\leq \int_{0}^{T}\int_{\Omega }\rho
\left( s+t\right) g\left( \partial _{t}\phi \right) \partial _{t}v+\partial
_{t}v\partial _{t}\phi d\mathfrak{m}_{a},
\end{equation*}%
Using $\left( \text{\ref{contrac inequ}}\right) $, we obtain%
\begin{eqnarray*}
\left\Vert Z\left( t+T\right) \varphi \right\Vert _{H} &\leq &\left(
\int_{0}^{T}\int_{\Omega }\rho \left( s+t\right) g\left( \partial _{t}\phi
\right) \partial _{t}v+\partial _{t}v\partial _{t}ud\mathfrak{m}_{a}\right)
^{1/2}+c_{0}\left\Vert Z\left( t\right) \varphi \right\Vert _{H} \\
&=&\left( \int_{0}^{T}\int_{\Omega }\rho \left( s+t\right) g\left( \partial
_{t}\phi \right) \partial _{t}v+\partial _{t}v\partial _{t}\phi d\mathfrak{m}%
_{a}\right) ^{1/2} \\
&&+c_{0}\left( \left\Vert Z\left( t\right) \varphi \right\Vert
_{H}-\left\Vert Z\left( t+T\right) \varphi \right\Vert _{H}\right)
+c_{0}\left\Vert Z\left( t+T\right) \varphi \right\Vert _{H}.
\end{eqnarray*}%
After some computation, we deduce that%
\begin{eqnarray*}
\left\Vert Z\left( t+T\right) \varphi \right\Vert _{H} &\leq &\frac{1}{%
1-c_{0}}\left( \int_{0}^{T}\int_{\Omega }\rho \left( s+t\right) g\left(
\partial _{t}\phi \right) \partial _{t}v+\partial _{t}v\partial _{t}\phi d%
\mathfrak{m}_{a}\right) ^{1/2} \\
&&+\frac{c_{0}}{1-c_{0}}\left( \left\Vert Z\left( t\right) \varphi
\right\Vert _{H}-\left\Vert Z\left( t+T\right) \varphi \right\Vert
_{H}\right) ,
\end{eqnarray*}%
this gives%
\begin{eqnarray*}
\left\Vert Z\left( t\right) \varphi \right\Vert _{H} &\leq &\frac{1}{1-c_{0}}%
\left( \int_{0}^{T}\int_{\Omega }\rho \left( s+t\right) g\left( \partial
_{t}\phi \right) \partial _{t}v+\partial _{t}v\partial _{t}\phi d\mathfrak{m}%
_{a}\right) ^{1/2} \\
&&+\left( \frac{1}{1-c_{0}}\right) \left( \left\Vert Z\left( t\right)
\varphi \right\Vert _{H}-\left\Vert Z\left( t+T\right) \varphi \right\Vert
_{H}\right) .
\end{eqnarray*}%
Using the fact that for $a>b>0,$ $a-b\leq \left( a^{2}-b^{2}\right) ^{1/2},$
we obtain%
\begin{equation*}
\begin{array}{cc}
\left\Vert Z\left( t\right) \varphi \right\Vert _{H} & \leq \frac{1}{1-c_{0}}%
\left( \int_{0}^{T}\int_{\Omega }\rho \left( s+t\right) g\left( \partial
_{t}\phi \right) \partial _{t}v+\partial _{t}v\partial _{t}\phi d\mathfrak{m}%
_{a}\right) ^{1/2} \\ 
& +\left( \frac{1}{1-c_{0}}\right) \left( \left\Vert Z\left( t\right)
\varphi \right\Vert _{H}^{2}-\left\Vert Z\left( t+T\right) \varphi
\right\Vert _{H}^{2}\right) ^{1/2}.%
\end{array}%
\end{equation*}%
Therefore we have%
\begin{equation*}
\begin{array}{cc}
\left\Vert Z\left( t\right) \varphi \right\Vert _{H}^{2}\leq  & 
C\int_{0}^{T}\int_{\Omega }\rho \left( s+t\right) g\left( \partial _{t}\phi
\right) \partial _{t}v+\partial _{t}v\partial _{t}\phi d\mathfrak{m}_{a} \\ 
& +C\left( \left\Vert Z\left( t\right) \varphi \right\Vert
_{H}^{2}-\left\Vert Z\left( t+T\right) \varphi \right\Vert _{H}^{2}\right) ,%
\end{array}%
\end{equation*}%
with $C=\frac{2}{\left( 1-c_{0}\right) ^{2}}.$

Since 
\begin{equation*}
U_{t}\left( s\right) Z\left( t\right) \varphi =U\left( s+t\right) \varphi ,%
\text{ on }B_{R},
\end{equation*}%
we obtain%
\begin{equation*}
\begin{array}{cc}
\left\Vert Z\left( t\right) \varphi \right\Vert _{H}^{2}\leq & 
C\int_{0}^{T}\int_{\Omega }\rho \left( s+t\right) g\left( \partial
_{t}u\left( s+t\right) \right) \partial _{t}v+\partial _{t}v\partial
_{t}u\left( s+t\right) d\mathfrak{m}_{a} \\ 
& +C\left( \left\Vert Z\left( t\right) \varphi \right\Vert
_{H}^{2}-\left\Vert Z\left( t+T\right) \varphi \right\Vert _{H}^{2}\right) .%
\end{array}%
\end{equation*}
\end{proof}

In the next we state some auxiliary results which will be used in the proof
of Theorem 1. More precisely, arguing as in \cite{mid} we estimate the first
term in the right hand side of the mixed observability estimate $\left( \ref%
{mixed observability}\right) $.

\begin{lemma}
\label{lem:concave estimates global} Let $t,T\geq 0$ and setting 
\begin{equation*}
\Omega _{t}^{0}=\left\{ \left( s,x\right) \in \left[ t,t+T\right] \times
\Omega ;\text{ }\left\vert \partial _{s}u\left( s,x\right) \right\vert <\eta
_{0}\right\} ,\quad \Omega _{t}^{1}=\Omega _{t,t+T}\setminus \Omega _{t}^{0}.
\end{equation*}%
We define 
\begin{equation*}
\Theta \left( \Omega _{t}^{i}\right) =\int_{\Omega _{t}^{i}}\left\vert
\partial _{s}u\left( s\right) \partial _{s}v\left( s-t\right) \right\vert d%
\mathfrak{m}_{a},\quad \Psi \left( \Omega _{t}^{i}\right) =\int_{\Omega
_{t}^{i}}\left\vert g\left( \partial _{s}u\left( s\right) \right) \partial
_{s}v\left( s-t\right) \right\vert d\mathfrak{m}_{a},\text{ for }i=0,1,
\end{equation*}%
where $u$ and $v$ denote respectively the solution of $\left( \text{\ref%
{sys:nonlinear}}\right) $ and $\left( \text{\ref{sys linear}}\right) $ with
initial data $\left( \varphi _{0},\varphi _{1}\right) $ and $\left( \left(
v_{0},v_{1}\right) =Z\left( t\right) \varphi \right) $\ in $H$.

There exists a positive constant $C$ which may depend on $T$ such that the
following inequalities hold for every $\epsilon >0$ :

\begin{enumerate}
\item (Estimate on the damping near the origin) 
\begin{equation}
\medskip 
\begin{array}{l}
\Theta \left( \Omega _{t}^{0}\right) +\Psi \left( \Omega _{t}^{0}\right)
\leq \epsilon \left\Vert Z\left( t\right) \varphi \right\Vert _{H}^{2} \\ 
+C\left( 1+\frac{1}{\epsilon }\right) \mathfrak{m}_{a}\left( \Omega \right)
h_{0}\left( {\frac{1}{\mathfrak{m}_{a}\left( \Omega \right) }}\int_{\Omega
_{t,t+T}}g\left( \partial _{s}u\right) \partial _{s}u\;d\mathfrak{m}%
_{a}\right) .%
\end{array}
\label{Omega zero estimate global}
\end{equation}

\item (Estimate on the damping near infinity) 
\begin{equation}
\begin{array}{l}
\Theta \left( \Omega _{t}^{1}\right) +\Psi \left( \Omega _{t}^{1}\right)
\leq \epsilon \left\Vert Z\left( t\right) \varphi \right\Vert
_{H}^{2}+C\epsilon ^{-1}\int_{\Omega _{t,t+T}}g\left( \partial _{s}u\right)
\partial _{s}u\;d\mathfrak{m}_{a}.%
\end{array}
\label{theta omega 1 estimate global}
\end{equation}
\end{enumerate}
\end{lemma}

\begin{remark}
\label{remark summarized estimate}It is easy to see that 
\begin{equation}
\Theta \left( \Omega _{t}^{0}\right) +\Psi \left( \Omega _{t}^{0}\right)
+\Theta \left( \Omega _{t}^{1}\right) +\Psi \left( \Omega _{t}^{1}\right)
\leq \epsilon \left\Vert Z\left( t\right) \varphi \right\Vert
_{H}^{2}+C\left( 1+\frac{1}{\epsilon }\right) h\left( \int_{\Omega
_{t,t+T}}g\left( \partial _{s}u\right) \partial _{s}ud\mathfrak{m}%
_{a}\right) ,  \label{summarized estimate lemma}
\end{equation}%
for every $\epsilon >0$.
\end{remark}

\begin{proof}
\textbf{Case (1). }Using Young's inequality, we obtain, for every $\epsilon
>0$%
\begin{equation*}
\begin{array}{l}
\Theta \left( \Omega _{t}^{0}\right) +\Psi \left( \Omega _{t}^{0}\right)
\leq \frac{1}{\epsilon }\int_{\Omega _{t}^{0}}\left[ \left\vert \partial
_{t}u\right\vert ^{2}+\left\vert g\left( \partial _{t}u\right) \right\vert
^{2}\right] d\mathfrak{m}_{a}+\epsilon \int_{\Omega _{t}^{0}}\left\vert
\partial _{t}v\left( s-t\right) \right\vert ^{2}d\mathfrak{m}_{a}.%
\end{array}%
\end{equation*}%
The energy estimate yields%
\begin{equation*}
\begin{array}{l}
\Theta \left( \Omega _{t}^{0}\right) +\Psi \left( \Omega _{t}^{0}\right)
\leq \frac{1}{\epsilon }\int_{\Omega _{t}^{0}}\left[ \left\vert \partial
_{t}u\right\vert ^{2}+\left\vert g\left( \partial _{t}u\right) \right\vert
^{2}\right] d\mathfrak{m}_{a}+\epsilon \left\Vert Z\left( t\right) \varphi
\right\Vert _{H}^{2}%
\end{array}%
\end{equation*}%
Using the inequality (\ref{h0 definition}) for the function $h_{0}$ we
obtain, 
\begin{equation*}
\int_{\Omega _{t}^{0}}\left[ g\left( \partial _{t}u\right) ^{2}+\left\vert
\partial _{t}u\right\vert ^{2}\right] d\mathfrak{m}_{a}\leq \frac{1}{%
\epsilon _{0}}\int_{\Omega _{t}^{0}}h_{0}\left( g\left( \partial
_{t}u\right) \partial _{t}u\right) d\mathfrak{m}_{a}.
\end{equation*}%
Now, since $h_{0}$ is concave, we use Jensen's inequality and we obtain%
\begin{equation*}
\begin{array}{cc}
\int_{\Omega _{t}^{0}}\left( g\left( \partial _{t}u\right) ^{2}+\left\vert
\partial _{t}u\right\vert ^{2}\right) d\mathfrak{m}_{a} & \leq \frac{T%
\mathfrak{m}_{a}\left( \Omega \right) }{\epsilon _{0}}h_{0}\left( {\frac{1}{T%
\mathfrak{m}_{a}\left( \Omega \right) }}\int_{\Omega _{t}^{0}}g\left(
\partial _{t}u\right) \partial _{t}u\;d\mathfrak{m}_{a}\right) \\ 
& \leq \frac{T\mathfrak{m}_{a}\left( \Omega \right) }{\epsilon _{0}}%
h_{0}\left( {\frac{1}{T\mathfrak{m}_{a}\left( \Omega \right) }}\int_{\Omega
_{t,t+T}}g\left( \partial _{t}u\right) \partial _{t}u\;d\mathfrak{m}%
_{a}\right) .%
\end{array}%
\end{equation*}%
Finally, by combining the above estimates we deduce (\ref{Omega zero
estimate global}).

\textbf{Case (2).} Applying Young's inequality and using $\left( \ref%
{linearly definition}\right) $,\ we find%
\begin{equation*}
\begin{array}{l}
\Theta \left( \Omega _{t}^{1}\right) +\Psi \left( \Omega _{t}^{1}\right)
\leq \epsilon \left\Vert Z\left( t\right) \varphi \right\Vert
_{H}^{2}+C\epsilon ^{-1}\int_{\Omega _{t,t+T}}g(\partial _{s}u)\partial
_{s}ud\mathfrak{m}_{a},%
\end{array}%
\end{equation*}%
for every $\epsilon >0$. This concludes the proof of \ the lemma.
\end{proof}

Before giving the proof of Theorem \ref{t:1}, we give the following Lemma
which is a time dependent version of the result in \cite[lemma 3.3]{las-tat}.

\begin{lemma}[\protect\cite{daou 2}]
\label{lemma las tat}Let

\begin{itemize}
\item $W\left( t\right) $ be a continuous, positive non-increasing function
for $t\in \mathbb{R}_{+}$.

\item $\theta $ is a non negative function $%
\mathbb{R}
_{+}$. Let $T>0$ and setting, $\kappa \left( t\right) =T\underset{\left[
t,t+T\right] }{\sup }\theta \left( s\right) $, $t\geq 0$.

\item Suppose for every $t\geq 0$, the functions $I-\kappa \left( t\right) 
\mathcal{L}\left( t,.\right) :\left[ 0,W\left( 0\right) \right] \rightarrow 
\mathbb{R}_{+}$ and $\mathcal{L}\left( t,.\right) :\left[ 0,W\left( 0\right) %
\right] \rightarrow \mathbb{R}_{+}$ are increasing, with $\mathcal{L}\left(
t,0\right) =0$

\item The function $\mathcal{L}\left( .,x\right) :\mathbb{R}_{+}\rightarrow 
\mathbb{R}_{+}$ is decreasing for every $x$ in $\left[ 0,W\left( 0\right) %
\right] $.

\item We assume that the following inequality%
\begin{equation}
\begin{array}{l}
W\left( \left( m+1\right) T\right) +\kappa \left( mT\right) \mathcal{L}%
\left( mT,W\left( mT\right) \right) \leq W\left( mT\right)%
\end{array}
\label{observability lemma}
\end{equation}%
holds, for $m=0$,$1$,$2$,... Moreover, $\mathcal{L}\left( t,s\right) $ does
not depend on $m$.
\end{itemize}

Then%
\begin{equation*}
W\left( t\right) \leq S\left( t-T\right) ,\qquad \forall t\geq T,
\end{equation*}%
where $S\left( t\right) $ is the solution of the following nonlinear
differential equation%
\begin{equation}
\begin{array}{l}
\frac{dS}{dt}+\theta \left( t\right) \mathcal{L}\left( t,S\left( t\right)
\right) =0;\qquad S(0)=W(0).%
\end{array}
\label{Ode lema}
\end{equation}%
Moreover, if there exists $T_{0}>>1$ such that%
\begin{equation}
\begin{array}{l}
\underset{t\rightarrow +\infty }{\lim }\int_{T_{0}}^{t}\theta \left(
s\right) \mathcal{L}\left( s,\gamma \right) ds=+\infty ,%
\end{array}
\label{over under damping condition1}
\end{equation}%
for every $0<\gamma <<1$. Then%
\begin{equation*}
\underset{t\rightarrow +\infty }{\lim }S\left( t\right) =0.
\end{equation*}
\end{lemma}

For the proof of Lemma \ref{lemma las tat} we refer the reader to \cite{daou
2}.

\subsection{Proof of Theorem \protect\ref{t:1}}

Let $\varphi \in K$. For every $t\geq 0$, $Z\left( t\right) \varphi \in K$
and we have%
\begin{eqnarray}
\left\Vert Z\left( t\right) \varphi \right\Vert _{H}^{2} &\leq &C\left(
\left\Vert Z\left( t\right) \varphi \right\Vert _{H}^{2}-\left\Vert Z\left(
t+T\right) \varphi \right\Vert _{H}^{2}\right)  \notag \\
&&+C\int_{t}^{t+T}\int_{\Omega }\rho \left( s\right) g\left( \partial
_{t}u\right) \partial _{t}v\left( s-t\right) +\partial _{t}u\partial
_{t}v\left( s-t\right) d\mathfrak{m}_{a}  \notag \\
&\leq &C\left( \left\Vert Z\left( t\right) \varphi \right\Vert
_{H}^{2}-\left\Vert Z\left( t+T\right) \varphi \right\Vert _{H}^{2}\right) 
\notag \\
&&+C\int_{t}^{t+T}\int_{\Omega }\left[ \rho \left( s\right) g\left( \partial
_{t}u\left( s\right) \right) \partial _{s}\tilde{v}\left( s\right) +\partial
_{t}u\left( s\right) \partial _{s}\tilde{v}\left( s\right) \right] d%
\mathfrak{m}_{a}  \label{estimate 1}
\end{eqnarray}%
where $\tilde{v}=v\left( s-t\right) $, for $s\geq t$.

\begin{enumerate}
\item \textbf{The function }$\rho $\textbf{\ is increasing on }$%
\mathbb{R}
_{+}$.%
\begin{equation*}
\begin{array}{l}
\left\Vert Z\left( t\right) \varphi \right\Vert _{H}^{2}\leq C\rho \left(
t+T\right) \left[ \Theta \left( \Omega _{t}^{0}\right) +\Psi \left( \Omega
_{t}^{0}\right) +\Theta \left( \Omega _{t}^{1}\right) +\Psi \left( \Omega
_{t}^{1}\right) \right] \\ 
\text{ \ \ \ \ \ \ \ \ \ \ \ \ \ \ \ \ \ }+C\left( \left\Vert Z\left(
t\right) \varphi \right\Vert _{H}^{2}-\left\Vert Z\left( t+T\right) \varphi
\right\Vert _{H}^{2}\right) \\ 
\text{ \ \ \ \ \ \ \ \ \ \ \ \ \ \ \ }=:I_{1}+I_{2}.%
\end{array}%
\end{equation*}%
Using $\left( \ref{estimate 1}\right) $ and $\left( \text{\ref{summarized
estimate lemma}}\right) $ and taking $\epsilon =\frac{1}{4\rho \left(
t+T\right) }$ (where $\epsilon $ is the constant that appears in Lemma \ref%
{lem:concave estimates global}),we obtain 
\begin{equation*}
I_{1}\leq C_{T}\left( \rho \left( t+T\right) +\rho ^{2}\left( t+T\right)
\right) h\left( \int_{\Omega _{t,t+T}}g\left( \partial _{s}u\right) \partial
_{s}u\;d\mathfrak{m}_{a}\right) .
\end{equation*}%
Since $\rho \left( t\right) \geq 1$ for all $t\geq 0$, we infer that%
\begin{equation*}
\begin{array}{c}
I_{1}\leq \frac{C_{T}}{\alpha \left( t\right) }h\left( \int_{\Omega
_{t,t+T}}g\left( \partial _{s}u\right) \partial _{s}u\;d\mathfrak{m}%
_{a}\right) ,%
\end{array}%
\end{equation*}%
where the function $\alpha $ is defined in Section 1. Now, using $\left( \ref%
{system phi}\right) $ and $\left( \ref{zt estimate}\right) $, we find%
\begin{eqnarray*}
\int_{\Omega _{t,t+T}}g\left( \partial _{s}u\right) \partial _{s}ud\mathfrak{%
m}_{a} &=&\int_{0}^{T}\int_{\Omega }g\left( \partial _{s}\phi \right)
\partial _{s}\phi d\mathfrak{m}_{a} \\
&\leq &\frac{1}{\rho \left( t\right) }\int_{0}^{T}\int_{\Omega }\rho \left(
s+t\right) g\left( \partial _{s}\phi \right) \partial _{s}\phi d\mathfrak{m}%
_{a} \\
&\leq &\frac{1}{\rho \left( t\right) }\left( \left\Vert Z\left( t\right)
\varphi \right\Vert _{H}^{2}-\left\Vert U_{t}\left( T\right) Z\left(
t\right) \varphi \right\Vert _{H}^{2}\right) \\
&\leq &\frac{1}{\rho \left( t\right) }\left( \left\Vert Z\left( t\right)
\varphi \right\Vert _{H}^{2}-\left\Vert Z\left( t+T\right) \varphi
\right\Vert _{H}^{2}\right) .
\end{eqnarray*}%
Thus%
\begin{equation}
\begin{array}{c}
I_{1}\leq \frac{C_{T}}{\alpha \left( t\right) }h\left( \frac{1}{\rho \left(
t\right) }\left( \left\Vert Z\left( t\right) \varphi \right\Vert
_{H}^{2}-\left\Vert Z\left( t+T\right) \varphi \right\Vert _{H}^{2}\right)
\right) .%
\end{array}
\label{I1 estimate increasing}
\end{equation}%
For the term $I_{2}$ it is obvious that,%
\begin{equation*}
\begin{array}{c}
I_{2}\leq \frac{C_{T}}{\alpha \left( t\right) }\left( \frac{1}{\rho \left(
t\right) }\left( \left\Vert Z\left( t\right) \varphi \right\Vert
_{H}^{2}-\left\Vert Z\left( t+T\right) \varphi \right\Vert _{H}^{2}\right)
\right) .%
\end{array}%
\end{equation*}%
Combining the estimate above with $\left( \ref{I1 estimate increasing}%
\right) $, we obtain%
\begin{eqnarray*}
\left\Vert Z\left( t\right) \varphi \right\Vert _{H}^{2} &\leq &\frac{C_{T}}{%
\alpha \left( t\right) }\left( I+h\right) \left[ \frac{1}{\rho \left(
t\right) }\left( \left\Vert Z\left( t\right) \varphi \right\Vert
_{H}^{2}-\left\Vert Z\left( t+T\right) \varphi \right\Vert _{H}^{2}\right) %
\right] \\
&\leq &\frac{C_{T}}{\alpha \left( t\right) }h\left[ \frac{1}{\rho \left(
t\right) }\left( \left\Vert Z\left( t\right) \varphi \right\Vert
_{H}^{2}-\left\Vert Z\left( t+T\right) \varphi \right\Vert _{H}^{2}\right) %
\right] .
\end{eqnarray*}%
\ This yields,%
\begin{equation*}
\left\Vert Z\left( t+T\right) \varphi \right\Vert _{H}^{2}+\left( \underset{%
\left[ t,t+T\right] }{\sup }T\beta \left( s\right) \right) h^{-1}\left( 
\frac{\alpha \left( t\right) \left\Vert Z\left( t\right) \varphi \right\Vert
_{H}^{2}}{C_{T}}\right) \leq \left\Vert Z\left( t\right) \varphi \right\Vert
_{H}^{2}.
\end{equation*}

\item \textbf{The function }$\rho $\textbf{\ is decreasing on }$%
\mathbb{R}
_{+}$.

We follows the same computations as in case 1. However, we take in this case 
$\epsilon =\frac{1}{4}$ and we obtain 
\begin{equation*}
I_{1}\leq C_{T}h\left( \int_{\Omega _{t,t+T}}g\left( \partial _{s}u\right)
\partial _{s}u\;d\mathfrak{m}_{a}\right) .
\end{equation*}%
This implies,%
\begin{equation*}
\begin{array}{c}
I_{1}\leq \frac{C_{T}}{\alpha \left( t\right) }h\left( \int_{\Omega
_{t,t+T}}g\left( \partial _{s}u\right) \partial _{s}u\;d\mathfrak{m}%
_{a}\right)%
\end{array}%
\end{equation*}%
where the function $\alpha $ is defined in Section 1. Now using the fact
that $\rho $ is decreasing, $\left( \ref{system phi}\right) $ and $\left( %
\ref{zt estimate}\right) $ we find%
\begin{eqnarray*}
\int_{\Omega _{t,t+T}}g\left( \partial _{s}u\right) \partial _{s}ud\mathfrak{%
m}_{a} &=&\int_{0}^{T}\int_{\Omega }g\left( \partial _{s}\phi \right)
\partial _{s}\phi d\mathfrak{m}_{a} \\
&\leq &\frac{1}{\rho \left( t+T\right) }\int_{0}^{T}\int_{\Omega }\rho
\left( s+t\right) g\left( \partial _{s}\phi \right) \partial _{s}\phi d%
\mathfrak{m}_{a} \\
&\leq &\frac{1}{\rho \left( t+T\right) }\left( \left\Vert Z\left( t\right)
\varphi \right\Vert _{H}^{2}-\left\Vert U_{t}\left( T\right) Z\left(
t\right) \varphi \right\Vert _{H}^{2}\right) \\
&\leq &\frac{1}{\rho \left( t+T\right) }\left( \left\Vert Z\left( t\right)
\varphi \right\Vert _{H}^{2}-\left\Vert Z\left( t+T\right) \varphi
\right\Vert _{H}^{2}\right) .
\end{eqnarray*}%
Therefore%
\begin{equation}
\begin{array}{c}
I_{1}\leq \frac{C_{T}}{\alpha \left( t\right) }h\left( \frac{1}{\rho \left(
t+T\right) }\left( \left\Vert Z\left( t\right) \varphi \right\Vert
_{H}^{2}-\left\Vert Z\left( t+T\right) \varphi \right\Vert _{H}^{2}\right)
\right) .%
\end{array}
\label{I1 estimate decreasing}
\end{equation}%
Next we estimate the term as follow, using $\rho \left( t\right) \leq 1$,
for all $t\geq 0$,%
\begin{equation*}
\begin{array}{c}
I_{2}\leq \frac{C_{T}}{\alpha \left( t\right) }\left( \frac{1}{\rho \left(
t+T\right) }\left( \left\Vert Z\left( t\right) \varphi \right\Vert
_{H}^{2}-\left\Vert Z\left( t+T\right) \varphi \right\Vert _{H}^{2}\right)
\right) .%
\end{array}%
\end{equation*}%
Combining the estimate above with $\left( \ref{I1 estimate decreasing}%
\right) $, we obtain%
\begin{equation*}
\left\Vert Z\left( t\right) \varphi \right\Vert _{H}^{2}\leq \frac{C_{T}}{%
\alpha \left( t\right) }h\left[ \frac{1}{\rho \left( t+T\right) }\left(
\left\Vert Z\left( t\right) \varphi \right\Vert _{H}^{2}-\left\Vert Z\left(
t+T\right) \varphi \right\Vert _{H}^{2}\right) \right] .
\end{equation*}%
\ This gives,%
\begin{equation*}
\left\Vert Z\left( t+T\right) \varphi \right\Vert _{H}^{2}+\left( \underset{%
\left[ t,t+T\right] }{\sup }T\beta \left( s\right) \right) h^{-1}\left( 
\frac{\alpha \left( t\right) \left\Vert Z\left( t\right) \varphi \right\Vert
_{H}^{2}}{C_{T}}\right) \leq \left\Vert Z\left( t\right) \varphi \right\Vert
_{H}^{2},
\end{equation*}%
for every $t\geq 0$, where the functions $\alpha $ and $\beta $ are defined
in Subsection \ref{sec:h0}. We summarize that in both cases $\rho $
increasing and $\rho $ decreasing we obtained that 
\begin{equation}
\left\Vert Z\left( t+T\right) \varphi \right\Vert _{H}^{2}+\left( \underset{%
\left[ t,t+T\right] }{\sup }T\beta \left( s\right) \right) h^{-1}\left( 
\frac{\alpha \left( t\right) \left\Vert Z\left( t\right) \varphi \right\Vert
_{H}^{2}}{C_{T}}\right) \leq \left\Vert Z\left( t\right) \varphi \right\Vert
_{H}^{2},  \label{ctdefinition}
\end{equation}%
for every $t\geq 0$ and for some $C_{T}>1.$
\end{enumerate}

Let $K\geq C_{T}$ and consider $t=mT$,%
\begin{equation*}
\begin{array}{l}
\left\Vert Z\left( \left( m+1\right) T\right) \varphi \right\Vert
_{H}^{2}+\left( \underset{\left[ mT,\left( m+1\right) T\right] }{\sup }%
T\beta \left( s\right) \right) h^{-1}\left( \alpha \left( mT\right) \frac{%
\left\Vert Z\left( mT\right) \varphi \right\Vert _{H}^{2}}{K}\right) \leq
\left\Vert Z\left( mT\right) \varphi \right\Vert _{H}^{2}%
\end{array}%
\end{equation*}%
for all $m\in \mathbb{N}$.

Hence the above estimate and Lemma \ref{lemma las tat}, imply the energy
result in Theorem \ref{t:1}.

Now set $W\left( t\right) =E_{u}\left( t\right) $, $\theta \left( t\right)
=\beta \left( t\right) $\ and let 
\begin{equation*}
\mathcal{L}\left( t,x\right) =h^{-1}\left( \frac{\alpha \left( t\right) }{K}%
x\right) ,\text{ for }\left( t,x\right) \in 
\mathbb{R}
_{+}\times \left[ 0,W\left( 0\right) \right] .
\end{equation*}%
We recall that $\kappa \left( t\right) =T\underset{\left[ t,t+T\right] }{%
\sup }\theta \left( s\right) $. It is clear that, $0<\alpha \left( t\right)
\leq 1$ and $0\leq \alpha \left( t\right) \kappa \left( t\right) \leq 1$,
for every $t\geq 0$. Then for every $x_{1}>x_{2}$ in $\left[ 0,W\left(
0\right) \right] $%
\begin{equation*}
\begin{array}{l}
\medskip \left( I-\kappa \left( t\right) \mathcal{L}\right) \left(
t,x_{1}\right) -\left( I-\kappa \left( t\right) \mathcal{L}\right) \left(
t,x_{2}\right) \\ 
\medskip \geq \frac{x_{1}-x_{2}}{K}\left( K-\alpha \left( t\right) \kappa
\left( t\right) \frac{h^{-1}\left( \left( \alpha \left( t\right)
x_{1}\right) /K\right) -h^{-1}\left( \left( \alpha \left( t\right)
x_{2}\right) /K\right) }{\left( \left( x_{1}-x_{2}\right) \alpha \left(
t\right) \right) /K}\right) ,%
\end{array}%
\end{equation*}%
which yields,%
\begin{equation*}
\begin{array}{l}
\medskip \left( I-\kappa \left( t\right) \mathcal{L}\right) \left(
t,x_{1}\right) -\left( I-\kappa \left( t\right) \mathcal{L}\right) \left(
t,x_{2}\right) \\ 
\medskip \geq \frac{x_{1}-x_{2}}{K}\left( K-\dfrac{h^{-1}\left( \left(
\alpha \left( t\right) x_{1}\right) /K\right) -h^{-1}\left( \left( \alpha
\left( t\right) x_{2}\right) /K\right) }{\left( \left( x_{1}-x_{2}\right)
\alpha \left( t\right) \right) /K}\right) .%
\end{array}%
\end{equation*}%
Using the fact that $h^{-1}$ is positive and convex on $\left[ 0,W\left(
0\right) \right] $, we obtain 
\begin{equation*}
\begin{array}{c}
\left\vert \frac{h^{-1}\left( \left( \alpha \left( t\right) x_{1}\right)
/K\right) -h^{-1}\left( \left( \alpha \left( t\right) x_{2}\right) /K\right) 
}{\left( \left( x_{1}-x_{2}\right) \alpha \left( t\right) \right) /K}%
\right\vert \leq f_{1}\left( \left( \alpha \left( t\right) E_{u}(0)\right)
/K\right) ,%
\end{array}%
\end{equation*}%
where%
\begin{equation*}
\begin{array}{cccc}
f_{1}: & 
\mathbb{R}
_{+} & \longrightarrow & 
\mathbb{R}
_{+} \\ 
& x & \longmapsto & \frac{d^{+}}{dx}h^{-1}\left( x\right) .%
\end{array}%
\end{equation*}%
Since $f_{1}$ is an increasing function on $%
\mathbb{R}
_{+}$ and $\alpha \left( t\right) \leq 1$, we only have to choose $K$ such
that%
\begin{equation*}
\begin{array}{l}
K\geq f_{1}\left( E_{u}(0)/K\right) \geq f_{1}\left( \left( \alpha \left(
t\right) E_{u}(0)\right) /K\right) .%
\end{array}%
\end{equation*}%
Knowing that $f_{1}\left( s\right) \leq 1$ for every $s$ and $K>C_{T}>1,$ we
deduce that the function, $\left( I-\kappa \left( t\right) \mathcal{L}\left(
t,.\right) \right) $ is increasing on $\left[ 0,W\left( 0\right) \right] $,
for every $t\geq 0$. Moreover it is clear that the function $t\rightarrow 
\mathcal{L}\left( t,x\right) $ is decreasing on $%
\mathbb{R}
_{+}$, for every $x$ in $\left[ 0,W\left( 0\right) \right] $ and $\mathcal{L}%
\left( t,.\right) :\left[ 0,W\left( 0\right) \right] \rightarrow \mathbb{R}%
_{+}$ is increasing for every $t\geq 0$. Now, from lemma \ref{lemma las tat}%
, we infer that 
\begin{equation*}
E_{u}\left( t\right) \leq S\left( t-T\right) ,\qquad \text{for all }t\geq T,
\end{equation*}%
where $S\left( t\right) $ is the solution of the following nonlinear
differential equation%
\begin{equation*}
\frac{dS}{dt}+\beta \left( t\right) h^{-1}\left( \frac{\alpha \left(
t\right) }{K}S\left( t\right) \right) =0,\qquad S(0)=E_{u}(0).
\end{equation*}%
Furthermore, if for some $T_{0}\gg 1$, the following result%
\begin{equation*}
\begin{array}{l}
\int_{T_{0}}^{t}\beta \left( s\right) h^{-1}\left( \frac{\gamma \alpha
\left( s\right) }{K}\right) ds\underset{t\rightarrow +\infty }{%
\longrightarrow }+\infty ,%
\end{array}%
\end{equation*}%
holds, for every $0<\gamma \ll 1$. We conclude that%
\begin{equation*}
\left\Vert Z\left( t\right) \varphi \right\Vert _{H}\underset{t\rightarrow
+\infty }{\longrightarrow }0.
\end{equation*}%
This completes the proof of Theorem \ref{t:1}.

\textbf{Acknowledgements: }\textit{The authors are very grateful to the
anonymous referees for their helpful comments and suggestions, that improved
the manuscript.}

\textit{This paper has been supported by deanship of scientific research of
Dammam University under the reference 2012045.}

\end{document}